\documentclass[12pt,a4paper,reqno]{amsart}
\usepackage{amsfonts}
\usepackage{amssymb}
\usepackage{amsthm}
\usepackage{amsmath}
\usepackage{newlfont}
\usepackage{graphicx}
\usepackage{color}

\def\r{\mathbb R}
\def\d{\mathrm d}
\def\l{\mathbb L}
\def\h{\mathbb H}

\def\t{\hbox{\bf T}}
\def\n{\hbox{\bf N}}
\def\b{\hbox{\bf B}}

\newtheorem{theorem}{Theorem}[section]
\newtheorem{definition}[theorem]{Definition}
\newtheorem{proposition}[theorem]{Proposition}

\newtheorem{example}[theorem]{Example} 
\newtheorem{lemma}[theorem]{Lemma} 

\title{Classification of separable surfaces with constant Gaussian curvature}
\author{Thomas Hasanis}
\address{$^1$ Department of Mathematics\\
               University of Ioannina\\
               45110 Ioannina, Greece}
\email{thasanis@cc.uoi.gr}

\author{Rafael L\'opez}
\address{$^2$ Departamento de Geometr\'{\i}a y Topolog\'{\i}a\\
 Universidad de Granada\\
 18071 Granada, Spain }
\email{rcamino@ugr.es}
\thanks{Rafael L\'opez has  partially supported by  the grant no. MTM2017-89677-P, MINECO/AEI/FEDER, UE.}


\date{}

\begin{document}
\maketitle

\begin{abstract}
We classify  all  surfaces  with constant Gaussian curvature $K$ in Euclidean $3$-space that can be expressed as an implicit equation of type $f(x)+g(y)+h(z)=0$, where $f$, $g$ and $h$ are real functions of one variable.   If $K=0$, we prove that the surface is a  surface of revolution, a cylindrical surface or a conical surface, obtaining explicit parametrizations of such surfaces. If $K\not=0$, we prove that the surface is a surface of revolution.  
\end{abstract}

\section{Introduction and statement of the results} 

The  objective of our investigation is the classification of all surfaces with constant Gaussian curvature in Euclidean $3$-space  that can be expressed by  an implicit equation of type $f(x)+g(y)+h(z)=0$, where $f$, $g$ and $h$ are real functions of one variable. Our   motivation arises from the classical theory of minimal surfaces. For example,  historically the first two minimal surfaces are separable, namely, the catenoid   $\cosh(z)^2=x^2+y^2$   by Euler in 1744, and    the helicoid $\tan(z)=y/x$   by Meusnier 1776.

 In 1835,   Scherk discovered all minimal surfaces of type $z=\phi(x)+\psi(y)$, where $\phi$ and $\psi$ are two real functions (\cite{sc}). Later,   Weingarten addressed the classification problem of all minimal surfaces of type $f(x)+g(y)+h(z)=0$,  realizing  that  form a rich and large family of minimal  surfaces (\cite{we}). For example, this family contains a variety of  minimal surfaces given in term of elliptic integrals as well as  periodic minimal surfaces such as the  Schwarz surfaces of type $P$ and $D$.   In the middle of the above century,   Fr\'echet  gave a deep study of these surfaces obtaining   examples with explicit parametrizations \cite{fr1,fr2}.  The   reader can see a  description of these surfaces in  \cite[II-5.2]{ni}.   

We introduce the following terminology. Let $\r^3$ denote the Euclidean $3$-dimensional space, that is, the real vector $3$-space $\r^3$ endowed with the Euclidean metric $\langle,\rangle=dx^2+dy^2+dz^2$, where $(x,y,z)$ stand for the canonical coordinates of $\r^3$.  Locally, any surface of $\r^3$ is   the zero level set  $F(x,y,z)=0$ of a  function $F$ defined in  an open set $O\subset\r^3$. Our  interest  are   those surfaces where the function $F$ is a separable function of its variables $x$, $y$ and $z$.

\begin{definition} A  (regular) surface $S$ in $\r^3$ is said to be separable if  can be expressed as  
\begin{equation}\label{s1}
S=\{(x,y,z)\in\r^3:f(x)+g(y)+h(z)=0\}.
\end{equation}
 Here $f$, $g$ and $h$ are   smooth functions defined in some intervals $I_1$, $I_2$ and $I_3$ of $\r$, respectively. 
  \end{definition}
 By the regularity of $S$,    $f'(x)^2+g'(y)^2+h'(z)^2\not=0$ for every $x\in I_1$, $y\in I_2$ and $z\in I_3$.

 In this paper, we consider the following question:
 
\begin{quote} 
{\it What are the    separable surfaces in Euclidean space $\r^3$ with   constant Gaussian curvature?}
\end{quote}

There are  three particular examples of separable surfaces that deserve be pointed out because they are obtained by simples choices of the functions $f$, $g$ and $h$ in the  equation \eqref{s1}.
 
 \begin{enumerate}
\item {\it  Right cylinders}. A right cylinder is   formed by  all the lines that are orthogonal  to a given   planar curve $C$.  If   $C$ is contained in one coordinate plane, then the surface is separable where one of the functions $f$, $g$ or $h$ is constant. So, if $C$ is contained in the $xy$-plane, then  $C$ writes as $f(x)+g(y)=0$. The corresponding right cylinder is the surface $\{(x,y,z)\in\r^3: f(x)+g(y)=0\}$. 

 \item {\it Translation surfaces}. A translation surface is a surface that, after renaming the the coordinates, can be  expressed as $z=\phi(x)+\psi(y)$, where $\phi$ and $\psi$ are  smooth functions. This surface is the sum of the plane curves $x\mapsto (x,0,\phi(x))$   and   $z\mapsto (0,y,\psi(y))$, so  the surface is generated when we move one of  these curves by means of the translations along the other ones. A separable surface is a translation surface if and only one of the three functions in (\ref{s1}) is linear, say,   $h(z)=az+b$, with $a,b\in\r$ and $a\not=0$ (the case $a=0$ corresponds with a right cylinder). 
 
 \item {\it Rotational surfaces}. A surface of revolution  with  respect to   the $z$-line  writes as $h(z)=x^2+y^2$, hence that  it  is a separable surface. In general, if the rotation axis is parallel to  the $z$-line, the implicit equation of the surface is $h(z)=x^2+y^2+ax+by+c$, with $a,b,c\in\r$.
\end{enumerate}

Among the above surfaces, our question  has a known answer. Indeed, any right cylinder   has  zero constant Gaussian curvature. On the other hand, translation surfaces $z=\phi(x)+\psi(y)$ with constant Gaussian curvature were classified by Liu, proving that   $K=0$ and one of the functions $\phi$ or $\psi$ is linear (\cite{li}). In such a case, and if $\phi(x)=ax+b$, $a,b\in\r$, $a\not=0$, the implicit equation of the surface is $z=ax+b+\phi(y)$, hence that the surface is a   non-right cylindrical surface.  Finally, rotational surfaces with constant Gaussian curvature are well known in the literature; see for example \cite{gr}.

 A first  approach to the study of separable surfaces with constant Gaussian curvature was done by the second author and Moruz in \cite{lm}, where were   classified all surfaces of type  $z=\phi(x)\psi(y)$ with zero constant Gaussian curvature: these  surfaces  will appear as a particular case of  Theorem \ref{t1} below.

The purpose of this paper is to provide   a complete classification of all separable surfaces with constant Gaussian curvature, obtaining explicit parametrizations of these surfaces. The classification depends on if the Gaussian curvature is zero or is not zero.

\begin{theorem}\label{t1}
 The only  separable surfaces with zero constant Gaussian curvature are congruent to:
\begin{enumerate}
\item A right cylinder over a planar curve contained in one of the coordinate planes.  
\item A translation surface $z=ax+g(y)$, where $a\not=0$ and $g$ is any smooth function.
\item A surface of revolution with zero constant Gaussian curvature.  
\item A   cylindrical surface and a conical surface with   parametrizations \eqref{k01}, \eqref{k03} and \eqref{k02}; see Section \ref{s3}. 
\end{enumerate}
\end{theorem}

In case that $K$  is a nonzero constant, the  classification  is the following.

\begin{theorem} \label{t2}
The only separable surfaces with  nonzero constant Gaussian curvature are the rotational surfaces of constant curvature whose rotational axis is parallel to one of the coordinate axes.
\end{theorem}

The organization of this paper is as follows. In Section \ref{s2} we obtain the expression of the Gaussian curvature $K$ of a separable surface and we distinguish the above three  special cases  of separable surfaces. In Section \ref{s3} we give the proof of Theorem \ref{t1} and in Section \ref{s4} we prove Theorem \ref{t2}.

\section{Preliminaries}\label{s2}

In our study, we need   the expression of the Gaussian curvature $K$ for a   surface defined by the implicit equation $F(x,y,z)=0$ for a smooth function   $F$   defined in a domain of $\r^3$.   Although this calculation may be seen as a mere exercise, as far as the authors know, the first reference where appears such a computation is   \cite{do}. In fact, Dombrowski obtains the expression of   the Gauss-Kronecker curvature $K$ for a hypersurface in Euclidean space $\r^{n+1}$ given by an implicit function $F(x_1,\ldots,x_{n+1})=0$;  see also \cite{sp}. In \cite{do}, it was derived the following formula for $K$: 
$$K=\frac{1}{|\nabla F|^{n+2}}\nabla F^t\cdot\mbox{co(Hess)}(F) \cdot\nabla F,$$
 where $\nabla F$ is the gradient of $F$ and $\mbox{co(Hess)}(F)$ is the matrix formed by the cofactors of $\mbox{Hess}(F)$. In the Euclidean space $\r^3$, the above formula reduces into 
 
\begin{align*}K|\nabla F|^4=&F_x^2\left|\begin{array}{ll}F_{yy}&F_{yz}\\ F_{yz}&F_{zz}\end{array}\right| +F_y^2 \left|\begin{array}{ll}F_{zz}&F_{xz}\\ F_{xz}&F_{xx}\end{array}\right|+F_z^2\left|\begin{array}{ll}F_{xx}&F_{xy}\\ F_{xy}&F_{yy}\end{array}\right|\nonumber \\
&-2F_xF_y\left|\begin{array}{ll}F_{xy}&F_{yz}\\ F_{xz}&F_{zz}\end{array}\right|-2F_yF_z\left|\begin{array}{ll}F_{yz}&F_{xz}\\ F_{xy}&F_{xx}\end{array}\right|-2F_xF_z\left|\begin{array}{ll}F_{xz}&F_{xy}\\ F_{yz}&F_{yy}\end{array}\right|.\end{align*}
  If the surface is    defined by the implicit equation $f(x)+g(y)+h(z)=0$, then the above expression of $K$ simplifies in
\begin{equation}\label{k1}
f'^2g''h''+g'^2f''h''+h'^2f''g''=K(f'^2+g'^2+h'^2)^2.
\end{equation}

A first case to distinguish is when one of the functions $f$, $g$ or $h$ is constant. Without loss of generality, we suppose that $h$ is constant, so $h(z)=a$, $z\in I_3$, for some $a\in\r$. Then  the equation of the surface is $f(x)+g(y)+a=0$, that is,  the surface is a right cylinder over the plane curve $C=\{(x,y)\in\r^2: f(x)+g(y)+a=0\}$ and its Gaussian curvature is $K=0$.  This case is the item 1 in Theorem   \ref{t1}.

In   what follows, we suppose that $S$ is not a right cylinder over a plane curve contained in one of the three coordinate planes. This is equivalent to $f'(x)g'(y)h'(z)\not=0$ 
everywhere in $I_1\times I_2\times I_3$.  Thus we can  introduce   the new variables 
\begin{equation}\label{uvw}
u=f(x),\quad v=g(y),\quad w=h(z),
\end{equation}
which are related by the equation $u+v+w=0$ thanks to \eqref{s1}.  Define the functions
$$X(u)=f'(x)^2,\quad Y(v)=g'(y)^2,\quad Z(w)=h'(z)^2.$$
Then
$$X'(u)= 2f''(x), \quad Y'(v) =2g''(y),\quad Z'(w) =2h''(z).$$
With this change of variables, equation (\ref{k1}) becomes 
\begin{equation}\label{k2}
XY'Z'+YX'Z'+ZX'Y'=4K(X+Y+Z)^2,
\end{equation}
for all values $u$, $v$ and $w$ under the condition $u+v+w=0$.  

Throughout this paper we need to differentiate equations similar to \eqref{k2} involving functions depending on $u$, $v$ and $w$. Since these variables are not independent  because $u+v+w=0$, the following result will be useful in our computations.

\begin{lemma} \label{le1}
Let $Q=Q(u,v,w)$ be a smooth function defined in a domain $\Omega\subset\r^3$. If $Q(u,v,w)=0$ for any triple of the section $\Omega\cap\Pi$, where $\Pi$ is the plane of equation $u+v+w=0$, then on the section we have
$$Q_u=Q_v=Q_w,$$
where $Q_u$, $Q_v$ and $Q_v$ are the derivatives of $Q$ with respect to $u$, $v$ and $w$, respectively.
\end{lemma}
\begin{proof}
Since $w=-u-v$, then $Q(u,v,-u-v)=0$. Differentiating with respect to $u$, we deduce $Q_u-Q_w=0$.   Changing the roles of $u$, $v$ and $w$, we conclude the result. 
\qed\end{proof}

We need to distinguish the three special cases of separable surfaces described in the Introduction in terms of the functions $X$, $Y$ and $Z$.

\begin{proposition}\label{pr-tres}
With the above notation, we have:
\begin{enumerate}
\item If one of   the functions $X$, $Y$ or $Z$ vanishes, then the surface is a right cylinder over a planar curve contained in one of the three coordinates planes.
\item If one of the functions $X'$, $Y'$ or $Z'$ vanishes, then the surface is a right cylinder over a planar curve contained in one of the three coordinates planes or it is a translation surface.
\item If one of the functions $X'-Y'$,   $X'-Z'$ or $Y'-Z'$ vanishes, then the surface is one of the type studied in the previous item, or it is a surface of revolution whose rotation axis is parallel to one of the three coordinate axes.
\end{enumerate}
\end{proposition}

\begin{proof} We discuss case-by-case.
\begin{enumerate}
\item  When we introduced   the variables $u$, $v$ and $w$ in \eqref{uvw}, we showed that if one of the functions $f'$, $g'$ or $h'$ vanishes, then the surface is a right cylinder  over   a planar curve   contained in one of the three coordinate planes.
 
\item  Without loss of generality, we suppose  that  $Z'=0$. Because $Z=h'^2$, then   $h(z)=az+b$ with $a,b\in\r$.  This proves that the surface is a right cylinder ($a=0$) or a translation surface ($a\not=0$). 
\item  Without loss of generality, we suppose  that $X'-Y'=0$. Then there is $a\in\r$ such that $X'=Y'=a$. The case $a=0$ has been studied in the previous item. Assume now  $a\not=0$. Solving the equations $X'(u)=a$, $Y'(v)=a$, we find
$$f'(x)^2=af(x)+b_1,\quad g'(y)^2=a g(y)+b_2,$$
for some constants $b_1,b_2\in\r$. The solutions of these ODEs are
$$f(x)=\frac{(ax+c_1)^2}{4a}-\frac{b_1}{a},\quad g(y)=\frac{(ay+c_2)^2}{4a}-\frac{b_2}{a},$$
where $c_1,c_2\in\r$. Thus the surface is a surface of revolution with respect to a straight-line parallel to the $z$-axis.  
\end{enumerate}
\qed\end{proof}

\section{Case $K=0$: proof of Theorem \ref{t1}}\label{s3}

In this section we will study separable surfaces with zero constant Gaussian curvature and we will prove Theorem \ref{t1}. 
If the Gauss curvature $K$ is constantly zero, then  equation (\ref{k2}) becomes  
 
\begin{equation}\label{eq2}
XY'Z'+YX'Z'+ZX'Y'=0,\quad \mbox{for all } u+v+w=0.
\end{equation}

By Proposition \ref{pr-tres}, the cases that one of the functions $X$, $Y$ or $Z$ is constant, or $X'$, $Y'$ or $Z'$ is $0$, or $X'-Y'$, $Y'-Z'$ or $X'-Z'$ is $0$ corresponds with the items  1, 2 and 3 of Theorem \ref{t1}, respectively.

From now on, we will suppose that the surface is not of the above three cases. In particular, $X', Y', Z'\not=0$, so equation (\ref{eq2}) can be expressed as  
$$
X'Y'Z'\left(\frac{X}{X'}+\frac{Y}{Y'}+\frac{Z}{Z'}\right)=0,
$$
or equivalently, 
\begin{equation}\label{eq4}
\frac{X}{X'}+\frac{Y}{Y'}+\frac{Z}{Z'}=0.
\end{equation}
By using Lemma \ref{le1}, we differentiate with respect to $u$, $v$ and $w$, obtaining
$$\left(\frac{X}{X'}\right)'=\left(\frac{Y}{Y'}\right)'=\left(\frac{Z}{Z'}\right)'.$$
Because we have three functions depending in the variables $u$, $v$ and $w$,  there is $k\in\r$ such that 
\begin{equation}\label{eq44}
\left(\frac{X}{X'}\right)'=\left(\frac{Y}{Y'}\right)'=\left(\frac{Z}{Z'}\right)'=k.
\end{equation}
We distinguish two cases.

\begin{enumerate}
\item Case $k=0$. Because $XYZ\not=0$, we deduce that there are $a,b,c\in\r$, $a,b,c\not=0$, such that 
$$ \frac{X}{X'} =\frac{a}{2},\quad  \frac{Y}{Y'}=\frac{b}{2},\quad \frac{Z}{Z'}=\frac{c}{2},$$
with   $a+b+c=0$ because of (\ref{eq4}). The integration of these equations leads to
\begin{eqnarray*}
f(x)&=&-a\log(m_1x+n_1)\\
 g(y)&=&-b\log(m_2y+n_2)\\
 h(z)&=&-c\log(m_3z+n_3),
 \end{eqnarray*}
where $m_i,n_i\in\r$ and $m_i\not=0$, $1\leq i\leq 3$.
Thus the implicit equation of the surface   $f(x)+g(y)+h(z)=0$ becomes
\begin{equation}\label{ho1}
(m_1x+n_1)^{a}(m_2y+n_2)^{b}(m_3z+n_3)^{c}=1,
\end{equation}
 or equivalently   
$$\left(x+\frac{n_1}{m_1}\right)^{a}\left(y+\frac{n_2}{m_2}\right)^{b}=m_1^{-a}m_2^{-b}m_3^{a+b}\left(z+\frac{n_3}{m_3}\right)^{a+b}.$$
This surface is the generalized cone with apex $\left(-n_1/m_1, -n_2/m_2,-n_3/m_3\right)$ and the directrix is the planar curve 
$$\left\{
\begin{array}{ll}
& \left(x+\dfrac{n_1}{m_1}\right)^{a}\left(y+\dfrac{n_2}{m_2}\right)^{b}=m_1^{-a}m_2^{-b}m_3^{a+b}\left(d+\dfrac{n_3}{m_3}\right)^{a+b}\\
&z=d\not=-\dfrac{n_3}{m_3}.
\end{array}
\right.$$
This case is included in the item (4) of Theorem \ref{t1}. Equation (\ref{ho1}) can be  expressed as 
\begin{equation}\label{k01}
m_3z+n_3=(m_1x+n_1)^{p}(m_2y+n_2)^{q},\quad p+q=1.
\end{equation}
This surface is the graph of the product of two functions on the variables $x$ and $y$: see \cite{lm}.    The surface defined by the   equation (\ref{k01})  appeared in \cite[Th. 1.3]{lm}.

\item Case $k\not=0$. Replacing $k$ by $1/(2k)$ in (\ref{eq44}), and integrating, we deduce that there are $a,b,c\in\r$ such that
\begin{equation}\label{hoo}
\frac{X'}{X}=\frac{2k}{u+a},\quad \frac{Y}{Y'}=\frac{2k}{v+b},\quad \frac{Z}{Z'}=\frac{2k}{w+c}.
\end{equation}
Substituting in (\ref{eq4}),  
$$u+v+w+a+b+c=a+b+c=0,$$
because of $u+v+w=0$. Thus $a+b+c=0$. Integrating \eqref{hoo}, we find
$$f'(x)=m_1(f(x)+a)^k$$
$$g'(y)=m_2(g(y)+b)^k$$
$$h'(z)=m_3(h(z)+c)^k,$$
for   some nonzero real numbers $m_i$, $1\leq i\leq 3$. The solutions of the above equations depend if $k=1$ or $k\not=1$. 

\begin{enumerate}

\item Case $k=1$. There are $n_i\in\r$ such that
\begin{eqnarray*}
f(x)&=&n_1 e^{m_1x}-a\\
 g(y)&=&n_2e^{m_2y}-b\\
 h(z)&=&n_3e^{m_3 z}-c,
 \end{eqnarray*}
and thus the equation of the surface $S$ is 
\begin{equation}\label{k03}
n_1 e^{m_1x}+n_2e^{m_2y}+n_3e^{m_3 z}=0.
\end{equation}
This surface is a generalized cylinder whose generators are parallel to  $\left(1/m_1,1/m_2,1/m_3\right)$ and the directrix curve is
$$\left\{
\begin{array}{ll}
& \frac{x}{m_1}+\frac{y}{m_2}+\frac{z}{m_3}=0\\
&n_1 e^{m_1x}+n_2e^{m_2y}+n_3e^{m_3 z}=0.
\end{array}
\right.$$

\item Case $k\not=1$. There are $n_i\in\r$,  such that
\begin{eqnarray*}
f(x)&=& \left((1-k)(m_1x+n_1)\right)^{\frac{1}{1-k}}-a\\
 g(y)&= &\left((1-k)(m_2y+n_2)\right)^{\frac{1}{1-k}}-b\\
 h(z)&= &\left((1-k)(m_3z+n_3)\right)^{\frac{1}{1-k}}-c.
\end{eqnarray*}
Then the implicit equation of the surface   is
\begin{equation}\label{k02}
(m_1x+n_1) ^{\frac{1}{1-k}} +(m_2y+n_2)^{\frac{1}{1-k}}+(m_3z+n_3)^{\frac{1}{1-k}}=0.
\end{equation}
This surface is conical with apex the point $\left(-n_1/m_1,-n_2/m_2,-n_3/m_3\right)$ and the directrix  curve is
$$\left\{
\begin{array}{l}
 z=d\not=-\frac{n_3}{m_3} \\
(m_1x+n_1) ^{\frac{1}{1-k}} +(m_2y+n_2)^{\frac{1}{1-k}}=-(m_3d+n_3)^{\frac{1}{1-k}}.
\end{array}
\right.$$
It deserves to point out that for some values of $k$, as $k=(2n-1)/(2n)$, $n\in\mathbb{N}$, equation (\ref{k02}) represents only a point.

\end{enumerate}
\end{enumerate}

 The previous cases $k=1$ and $k\not=1$ correspond with the item 4 of Theorem \ref{t1} and this completes the proof.

We finish this section showing some explicit parametrizations  of separable surfaces with zero constant Gaussian curvature for the case 4  of Theorem \ref{t1}.

\begin{example} Case $k=0$ in equation (\ref{k01}). We choose  $p=2$, $q=-1$,   $m_i=1$ and $n_i=0$. Then the surface is $x^2=yz$. See figure \ref{fig1}, left.

\end{example}

\begin{example} Case $k=1$ in  equation (\ref{k03}). We choose   $n_1=-1$, $n_2=n_3=1$ and $m_i=1$.  The implicit equation of the surface is 
$-e^x+e^y+e^z=0$. See figure   \ref{fig1}, middle.
 \end{example}
 
\begin{example} Case $k=2$ in equation (\ref{k02}). We choose   $m_i=1$ and $n_i=0$, obtaining
$1/x+1/y+1/z=0$, or equivalently, $z=xy/(x+y)$ with $x+y\not=0$. See figure \ref{fig1}, right.
\end{example}

\begin{figure}[hbtp]
\begin{center}\includegraphics[width=.3\textwidth]{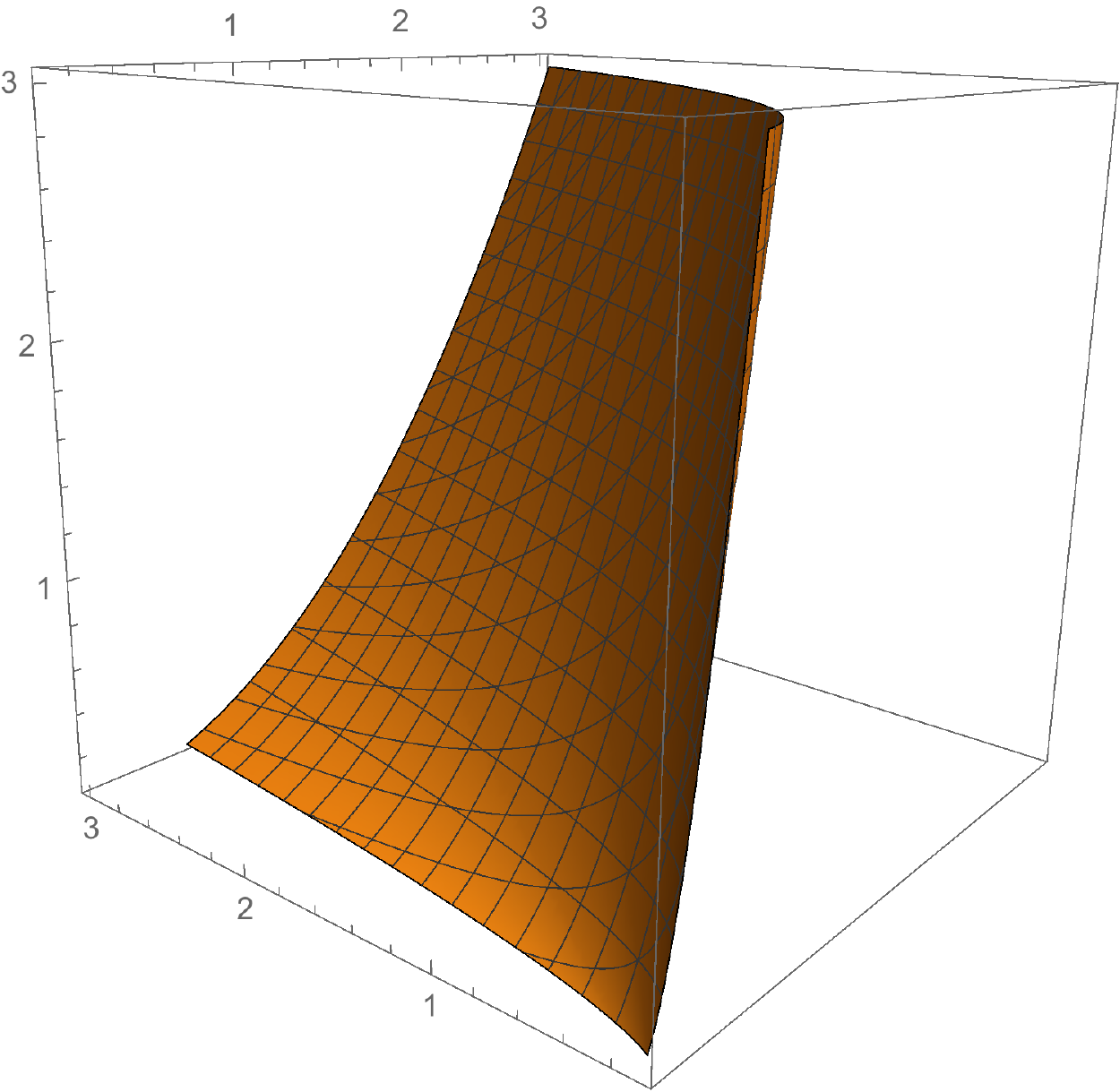}\quad \includegraphics[width=.3\textwidth]{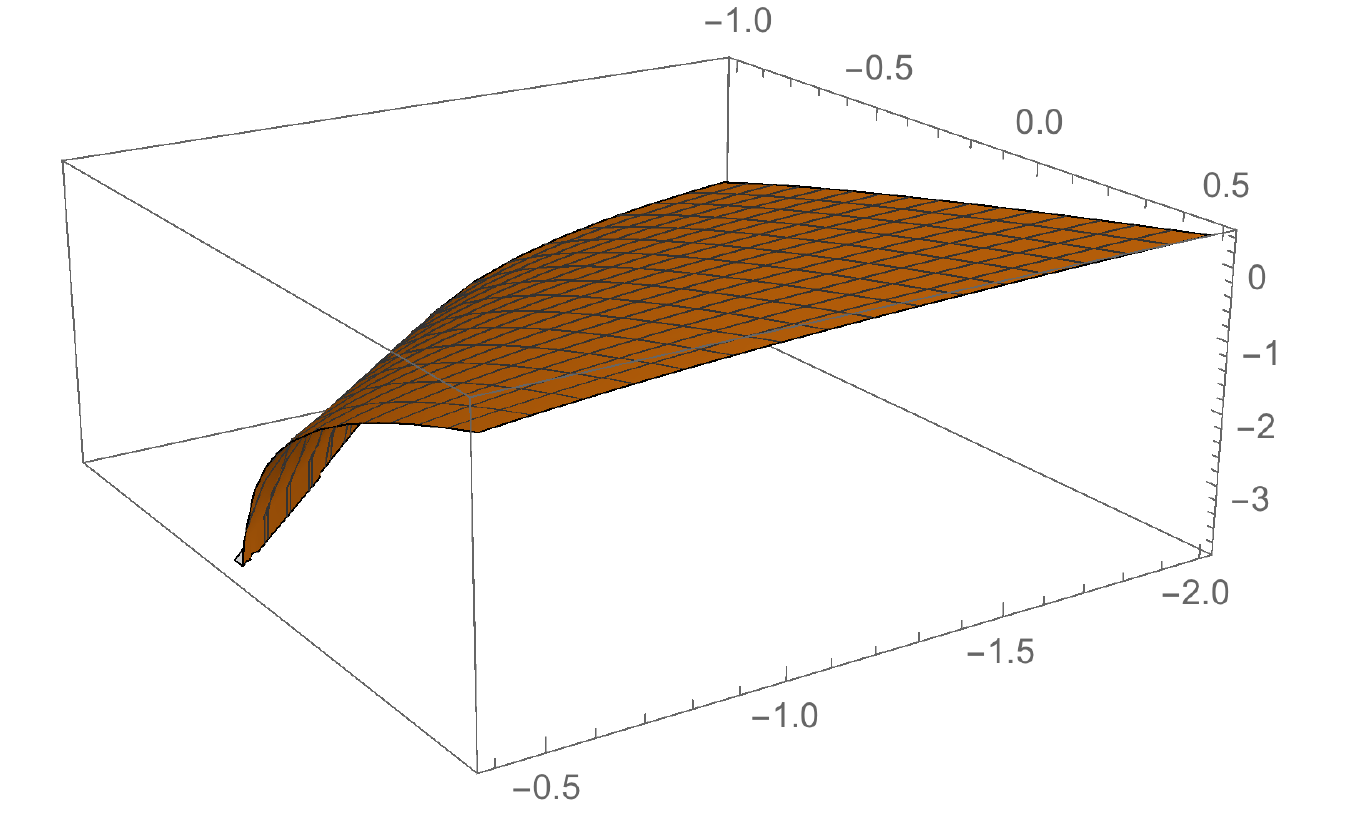}
\quad \includegraphics[width=.3\textwidth]{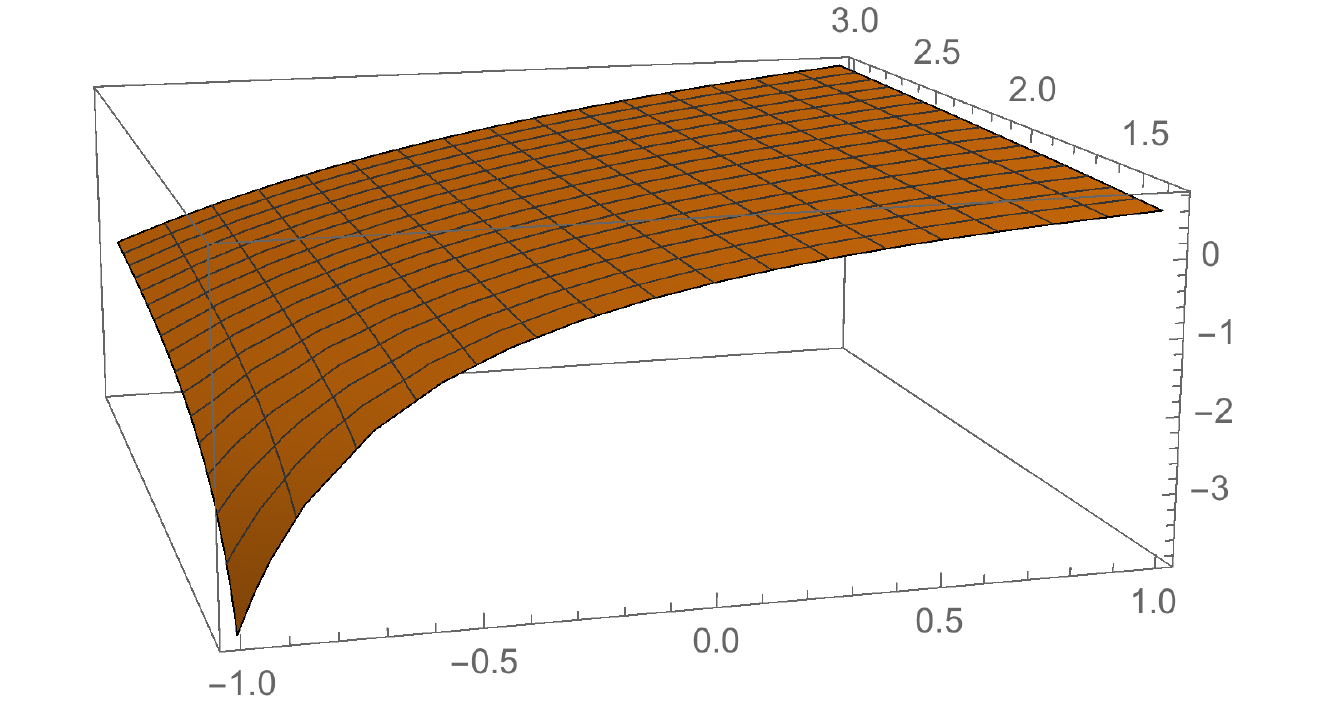} \end{center}
\caption{Left: the surface $x^2=yz$. Middle: the surface $-e^x+e^y+e^z=0$. 
Right:  the surface  $z=xy/(x+y)$.  }\label{fig1}
\end{figure}

\section{Case $K\not=0$: proof of Theorem \ref{t2}}  \label{s4}
In this section   we   prove Theorem \ref{t2}.  Consider a separable surface defined by (\ref{s1}) and suppose that the Gaussian curvature $K$ is a nonzero constant. In this setting, the particular cases that one of the functions $X$, $Y$ or $Z$ is constant, or that $X'$, $Y'$ or $Z'$ is zero and described in Proposition \ref{pr-tres}, can not appear because in such a case the Gaussian curvature should be zero. On the other hand, the case that one of the functions $X'-Y'$, $X'-Z'$ or $Y'-Z'$ vanishes proves that the surface is rotational about an axis parallel to one of the coordinate axes, proving just the statement of Theorem \ref{t2}. 

Therefore, and besides the rotational surfaces, it remains to prove that there are not  more surfaces of separable surfaces with nonzero constant Gaussian curvature. The proof of Theorem \ref{t2} is by contradiction.  Assume on the contrary  that the surface is not a surface of revolution about a straight-line parallel to one of the coordinates axes. In particular,  none of the functions $X'-Y'$, $X'-Z'$ or $Y'-Z'$ is identically $0$.  We write down again equation (\ref{k2}) 
\begin{equation}\label{eq50}XY'Z'+YX'Z'+ZX'Y'=4K(X+Y+Z)^2.
\end{equation}
Since $K\not=0$,   the surface is not a cylindrical surface neither a translation surface. This implies that $XYZ\not=0$ and $X'Y'Z'\not=0$. Then equation (\ref{eq50}) is 

\begin{equation}\label{eq5}
(X'Y+XY')Z'+(X'Y'-8K(X+Y))Z-4KZ^2-4K(X+Y)^2=0.
\end{equation} 

We simplify the notation of this equation by setting 
\begin{equation}\label{et21}
PZ'+QZ+R=4KZ^2,
\end{equation}
where
\begin{align*}
&P(u,v)=X'Y+XY'\\
&Q(u,v)=X'Y'-8K(X+Y)\\
&R(u,v)=-4K(X+Y)^2.
\end{align*}

Before to  indicate the arguments to prove Theorem \ref{t2}, we need a lemma that says us that the coefficient of $Z'$ in \eqref{eq5}, namely, the function $P$, is not zero.

\begin{lemma}\label{le12}
The function   $P=X'Y+XY'$ in (\ref{eq5}) can not vanish in any open set. 
\end{lemma}

\begin{proof} The proof  is by contradiction. If $X'Y+XY'=0$ in an open set, then
\begin{equation}\label{c11}
\frac{X'}{X}=a=-\frac{Y'}{Y}
\end{equation}
where $a\in\r$ is  a constant. Now equation (\ref{eq50}) is
\begin{equation}\label{c1}
-a^2XYZ=4K(X+Y+Z)^2,
\end{equation}
in particular, $a\not=0$. By applying Lemma \ref{le1} to this equation,  and  differentiating with respect to $u$ and $v$, we obtain  
\begin{equation}\label{c2}
-2a^3 XYZ=8Ka(X+Y)(X+Y+Z),
\end{equation}
where we  have used \eqref{c11}. By combining \eqref{c1} and \eqref{c2}, we conclude $8KaZ(X+Y+Z)=0$, arriving to  a contradiction. 
\qed\end{proof}

Once proved Lemma \ref{le2}, we return to the equation \eqref{et21}. We apply Lemma \ref{le1} by differentiating  (\ref{et21}) with respect to $u$ and $v$. Then we find
$$(P_u-P_v)Z'+(Q_u-Q_v)Z+R_u-R_v=0,$$
or equivalently
\begin{equation}\label{eq55}
 (X''Y-XY'')Z'+(X''Y'-X'Y''-8K(X'-Y'))Z-8K(X+Y)(X'-Y')=0.
\end{equation}
We set
\begin{equation}\label{pqr}
\begin{split}
A(u,v)&=P_u-P_v=X''Y-XY''\\
 B(u,v)&=Q_u-Q_v=X''Y'-X'Y''-8K(X'-Y')\\
C(u,v)&=R_u-R_v=-8K(X+Y)(X'-Y').
\end{split}
\end{equation}
Then (\ref{eq55}) is
\begin{equation}\label{eq6}
AZ'+BZ+C=0,\quad \mbox{for all } u+v+w=0.
\end{equation}

We now present the steps of the proof of Theorem \ref{t2}. We will prove two  technical lemmas. First, we show that the functions $A$, $B$ and $C$ can not vanish (Lemma \ref{le2}) and then  we show that   $B/A$ and $C/A$ are functions on the variable $u+v$ (Lemma \ref{le3}). After that, the proof returns to the equation \eqref{et21} and, after successive applications of Lemma \ref{le1}, we arrive to the desired contradiction.

\begin{lemma}\label{le2}
The coefficients $A$ or $B$ or $C$ in (\ref{eq6}) can not be identically zero.
\end{lemma}

\begin{proof}
It is clear that $C=0$ implies $X'-Y'=0$, which is not possible. By contradiction, we suppose $A=0$ or $B=0$ identically. 
\begin{enumerate}
\item Case $A=0$. Then $X''Y=XY''$, so $X''/X=Y''/Y=a$, where $a\in\r$ is a constant. Then (\ref{eq55}) is 
\begin{equation}\label{eq552}
F:=\left(a(XY'-X'Y)-8K(X'-Y')\right)Z-8K(X+Y)(X'-Y')=0.
\end{equation}
By Lemma \ref{le1}, the expression $F_u-F_v=0$ is  
$$G:=\left(2aX'Y'-2a^2XY-8aK(X+Y)\right)Z-8K((X'-Y')^2+a(X+Y)^2)=0.$$
Using Lemma \ref{le1} again for the function $G$, the equation  $G_u-G_v=0$ becomes
\begin{equation}\label{eq553}
 \left(4a^2(XY'-X'Y)-8Ka(X'-Y')\right)Z-32 aK(X+Y)(X'-Y')=0.
\end{equation}
Subtracting (\ref{eq553}) from (\ref{eq552}), we deduce
$$24aK(X'-Y')Z=0,$$
which it is a contradiction.

\item Case $B=0$. Then
$$\frac{X''}{X'}+8K\frac{1}{X'}=a=\frac{Y''}{Y'}+8K\frac{1}{Y'}$$
for some constant $a\in\r$. Then 
\begin{equation}\label{eq5-1}
X''=aX'-8K,\quad Y''=aY'-8K.
\end{equation}
Taking into account both expressions, equation (\ref{eq55}) is 
\begin{equation}\label{fff}
F:=\left(a(X'Y-XY')+8K(X-Y)\right)Z'-8K(X+Y)(X'-Y')=0.
\end{equation}
 We utilize Lemma \ref{le1} by differentiating with respect to $u$ and $v$. Then    $F_u-F_v=0$ and using (\ref{eq5-1}), we deduce
\begin{align*}G:&=\left(a^2(X'Y+XY')-8aK(X+Y)-2aX'Y'+8K(X'+Y')\right)Z' \\
& -8K\left((X'-Y')^2-8aK(X+Y)(X'+Y')+128K^2(X+Y)\right)=0.
\end{align*}
Again, we compute $G_u-G_v=0$ and using \eqref{eq5-1}, we have
\begin{eqnarray*}
&&\left(a^3(X'Y-XY')-16aK(X'-Y')+8a^2K(X-Y)\right)Z'\nonumber\\
&& -24aK(X'-Y')(X'+Y')-8a^2K(X+Y)(X'-Y')+384K^2(X'-Y')=0. 
\end{eqnarray*}
If $a=0$, then $X'-Y'=0$, which is not possible. So, $a\not=0$. By combining  this equation and (\ref{fff}), we deduce
$$
H:= 2aZ'+3a(X'-Y')-48K=0.
$$
Finally, using Lemma \ref{le1} for the function $H$ and \eqref{eq5-1}, the equation  $H_u-H_v=0$ yields 
$$3a(X''-Y'')=3a^2(X'-Y')=0,$$
 which is a contradiction.

\end{enumerate}
\qed\end{proof}

After Lemma \ref{le2}, we differentiate (\ref{eq6}) with respect to $u$ and $v$, obtaining
$$(A_u-A_v)Z'+(B_u-B_v)Z+(C_u-C_v)=0.$$

\begin{lemma}\label{le3}
With the above notation, we have  
$$\frac{A_u-A_v}{A}=\frac{B_u-B_v}{B}=\frac{C_u-C_v}{C},$$
or equivalently,
$$
\left(\frac{B}{A}\right)_u=\left(\frac{B}{A}\right)_v,\quad \left(\frac{C}{A}\right)_u=\left(\frac{C}{A}\right)_v,\quad \left(\frac{C}{B}\right)_u=\left(\frac{C}{B}\right)_v.
$$
\end{lemma}
 
\begin{proof}
 
 By eliminating $Z'$ from (\ref{eq5}) and (\ref{eq6}), we find
 \begin{eqnarray*}
 &&4KAZ^2+\Big\{ (8K(X+Y)-X'Y')A+(X'Y+XY')B\Big\}Z\\
&&+	4K(X+Y)^2A+(X'Y+XY')C=0.
\end{eqnarray*}
 We write this relation as
 \begin{equation}\label{eq10}
 LZ^2+MZ+N=0,
 \end{equation}
 where
 \begin{eqnarray*}
 L&=& 4KA\\
 M &=& \left(8K(X+Y)-X'Y'\right)A+(X'Y+XY')B\\
 N&=& 4K(X+Y)^2A+(X'Y+XY')C.
 \end{eqnarray*}
 
 Applying Lemma \ref{le1} to (\ref{eq10}) differentiating with respect to $u$ and $v$, we obtain
 \begin{equation}\label{eq11}
 (L_u-L_v)Z^2+(M_u-M_v)Z+N_u-N_v=0.
 \end{equation}
 Simple calculations give
 \begin{equation}\label{lu}
 \begin{split}
 &L_u-L_v=4K(A_u-A_v)\\
 &M_u-M_v=(8K(X+Y)-X'Y')(A_u-A_v)+(X'Y+XY')(B_u-B_v)\\
&N_u-N_v=4K(X+Y)^2(A_u-A_v)+(X'Y+XY')(C_u-C_v).
 \end{split}
 \end{equation}
 
We distinguish two cases in the proof of Lemma \ref{le3}:
\begin{enumerate}
\item Case $L_u-L_v\not=0$. Since equations (\ref{eq10}) and (\ref{eq11}) have the same solutions, 
$$
\frac{L_u-L_v}{L}=\frac{M_u-M_v}{M}=\frac{N_u-N_v}{N},
$$
or equivalently, 
\begin{equation}\label{eq13}
\begin{split}
&M(L_u-L_v)=L(M_u-M_v)\\
& N(L_u-L_v)=L(N_u-N_v).
\end{split}
\end{equation}
Using the expressions of $L_u-L_v$, $M_u-M_v$ and $N_u-N_v$ in (\ref{lu}), the two equations of (\ref{eq13}) are, respectively,
$$B(X'Y+XY')(A_u-A_v)=A(X'Y+XY')(B_u-B_v)$$
$$C(X'Y+XY')(A_u-A_v)=A(X'Y+XY')(C_u-C_v).$$
Since $X'Y+XY'\not=0$ by Lemma \ref{le12}, if follows the result of   Lemma \ref{le3}.

\item Case  $L_u-L_v=0$. 
\begin{enumerate}
\item Subcase $M_u-M_v=0$. Then   \eqref{eq11} implies  $N_u-N_v=0$. From (\ref{lu}) we deduce  $A_u-A_v=B_u-B_v=C_u-C_v=0$, proving the result.
\item Subcase $M_u-M_v\not=0$. From (\ref{eq11}), we deduce that (\ref{eq10}) has a unique solution, namely, 
$$Z=-\frac{M}{2L}.$$
Then
$$\left(\frac{M}{L}\right)_u-\left(\frac{M}{L}\right)_v=0.$$
This implies $(M_u-M_v)L=0$, and we conclude $M_u-M_v=0$, a contradiction.

\end{enumerate}
\end{enumerate}
\qed\end{proof}

Once proved Lemmas \ref{le2} and \ref{le3}, we are in position to complete the proof of Theorem \ref{t2}.

\begin{proof}[of Theorem \ref{t2}]

From Lemma \ref{le3} we deduce that there exist functions $\Phi$ and $\Psi$ of one variable such that
$$	\Psi(u+v)=\frac{B(u,v)}{A(u,v)},\quad      \Phi(u+v)=\frac{C(u,v)}{A(u,v)}.$$
By Lemma \ref{le2}, $A\not=0$ and  we write (\ref{eq10}) as
\begin{equation}\label{et22}
Z^2+\frac{M}{L}Z+\frac{N}{L}=0, 
\end{equation}
This  a polynomial equation on $Z$. Then the two roots $Z_1(w)$ and $Z_2(w)$ of this equation satisfy
$$Z_1(w)+Z_2(w)=-\frac{M}{L},\quad Z_1(w) Z_2(w)=-\frac{N}{L}.$$
Since $u+v+w=0$, the functions $M/L$ and $N/L$ depend only on the variable $u+v$. Denote 
$$\Psi_1(u+v)=\frac{M(u,v)}{L(u,v)},\quad \Phi_1(u+v)=\frac{N(u,v)}{L(u+v)}.$$
Thus \eqref{et22} is now
$$
Z^2+\Psi_1 Z+\Phi_1 =0.
$$
Applying Lemma \ref{le1} to this equation with   respect to $w$ and $u$, we find
$$(2Z+\Psi_1)Z'-\Psi_1' Z-\Phi_1'=0.$$
From this equation and  \eqref{et21}, we   eliminate  $Z'$, obtaining
$$\left((2Z+\Psi_1)Q+P\Psi_1'\right)Z+(2Z+\Psi_1)R+P\Phi_1'=4KZ^2(2Z+\Psi_1),$$
or equivalently, 
\begin{equation}\label{ppp}
Z^3+\left(\frac{\Psi_1}{2}-\frac{Q}{4K}\right)Z^2-\left(\frac{P\Psi_1'+2R+Q\Psi_1}{8K}\right)Z-\frac{P\Phi_1'+R\Psi_1}{8K}=0.
\end{equation}
This is a polynomial equation on $Z$. Let $Z_1(w)$, $Z_2(w)$ and $Z_3(w)$ denote the three roots of this equation. Then the function
$$T(u+v)=Z_1(w)+Z_2(w)+Z_3(w)$$
is the opposite of the coefficient of  $Z^2$ in \eqref{ppp}, so 
$$T(u+v)=\frac{Q(u,v)}{4K}-\frac{\Psi_1(u+v)}{2}.$$
We apply  to this equation Lemma \ref{le1} differentiating with respect to the variables $u$ and $v$, obtaining
$$0=\frac{1}{4K}(Q_u-Q_v)=\frac{1}{4K} B$$
by \eqref{pqr}. This is   a contradiction by Lemma \ref{le2}, and this concludes the proof of Theorem \ref{t2}.
\qed\end{proof}

\section*{Acknowledgements} Rafael L\'opez has been partially supported by the grant no. MTM2017-89677-P, MINECO/AEI/FEDER, UE.


\end{document}